\documentclass[twocolumn,amsmath,amssymb,aps,secnumarabic,%
    nofootinbib,superscriptaddress]{revtex4-1}
  \usepackage{amsmath,amssymb,amsthm}
  \usepackage{mathrsfs}
  \usepackage{graphicx,color}

\newtheorem{theo}{Theorem}[]
\newtheorem{lemm}[theo]{Lemma}

\theoremstyle{definition}
\newtheorem{rema}[theo]{Remark}

\begin{document}

\title{Yet another short proof of Bourgain's distorsion estimate}

\author{Beno\^{\i}t R. Kloeckner}
\email{benoit.kloeckner@ujf-grenoble.fr}
\affiliation{Institut Fourier, Universit\'e de Grenoble I}

\begin{abstract}
We use a self-improvement argument to give a very short and elementary proof of the 
result of Bourgain saying that regular trees do not admit bi-Lipschitz embeddings
into uniformly convex Banach spaces. 
\end{abstract}

\maketitle

Let $T_n$ be the binary rooted tree of depth $n$ and let $c_B(A)$ denote the
distorsion of the metric space $A$ in $B$, that is to say the infimum
of all numbers $D$ such that there is a number $s$ and a map $\varphi:A\to B$
such that
\[ s d(x,y) \leqslant d(\varphi(x),\varphi(y)) \leqslant s D d(x,y)\]
for all $x,y\in A$.

The modulus of (uniform) convexity $\delta_X(\varepsilon)$
of a Banach space $X$ with norm $|\cdot|$ is defined as 
\[\inf\left\{1-\left|\frac{x+y}2\right| \,\middle|\,
  | x| = | y| =1 \ \textrm{and}\ | x-y|\geqslant \varepsilon\right\}\]
for $\varepsilon\in (0,2]$. The space
$X$ is said to be uniformly convex of type $p\geqslant 2$
if $\delta_X(\varepsilon)\geqslant c \varepsilon^p$
for some $c>0$. Note that in particular, for $p\in(1,\infty)$ the $L^p$ spaces
are uniformly convex of type $\max(2,p)$.

Our main goal is to prove the following as simply as possible.
\begin{theo}[Bourgain]\label{theo:Bourgain}
If $X$ is uniformly convex of type $p$ then
\[c_X(T_n) \gtrsim (\ln n)^{\frac1p}.\]
\end{theo}

Several proofs of this result have been given over the years, see 
notably \cite{Bourgain,Matousek,Mendel-Naor}. As we discovered after writing a first draft of this
paper, the method we use is very close to Johnson and Schechtman's
proof of the distorsion estimate for diamond graphs \cite{Johnson-Schechtman}. However, 
it seems not to have been noticed before that this method gives such
an effective proof of Bourgain's estimate.

%
%
%

\begin{proof}
The first step is similar to previous proofs, notably
the one by Matou{\v{s}}ek's \cite{Matousek}.
Let $Y$ be the four-vertices
tree with one root $a_0$ which has one child $a_1$ and two grand-children
$a_2$, $a'_2$.

\begin{lemm}\label{lemm:geo}
There is a constant $K=K(X)$ such that
if $\varphi:Y\to X$ is $D$-Lipschitz and distance non-decreasing, then either
\[|\varphi(a_0)-\varphi(a_2)|\leqslant 2(D-\frac{K}{D^{p-1}})\]
or
\[|\varphi(a_0)-\varphi(a'_2)|\leqslant 2(D-\frac{K}{D^{p-1}})\]
\end{lemm}

We provide the proof below for the sake of completeness.

Let now $\varphi:T_n\to X$ a $D$-Lipschitz, distance non-decreasing map.
By the lemma, the root $a_0$ has at least two grand-children
$a_2^i$ ($i=1,2$) such that
\[2\leqslant |\varphi(a_0)-\varphi(a^i_2)|\leqslant 2f(D)\]
where $f(D)=D-K/D^{p-1}$.
Applying the lemma again, each of $a_2^i$
also has two grand-children satisfying similar inequalities, and we can apply he same
reasoning every other generation.
Restricting $\varphi$ to these vertices, we get an embedding
of $T_{\lfloor \frac{n}{2}\rfloor}$ whose distorsion is at most
$f(D)$.

We can iterate these restrictions $\lfloor\log_2(n)\rfloor$ times
to get an embedding of $T_1$ whose distorsion is $f^{\lfloor\log_2(n)\rfloor}(D)$.
This must be at least $1$ and $f^{D^p/K}(D)<1$, so that
\[\log_2(n) \lesssim D^p\]
which is Theorem \ref{theo:Bourgain}.
\end{proof}

\begin{rema}
Working out the constants gives the more precise result that 
\begin{equation}
c_X(T_n) \geqslant \left(\frac{pc}2\right)^{\frac1p}(\log_2 n)^{\frac1p} + \mbox{l.o.t.}
\label{eq:tree}
\end{equation}
where $c$ can be replaced by $\liminf \delta_X(\varepsilon)\varepsilon^{-p}$. In particular
\[c_{\ell^2}(T_n) \geqslant \frac{1}{2\sqrt{2}}(\log_2 n)^{\frac12}+ \mbox{l.o.t.}\]
\end{rema}

\begin{proof}[Proof of Lemma \ref{lemm:geo}]
Assume $\varphi(a_0)=0$ and let $x_1=\varphi(a_1)$, $x_2=\varphi(a_2)$ and
 $x'_2=\varphi(a'_2)$.

Suppose that $|x_2|\geqslant 2(D-\eta)$ for some $\eta$ to be chosen afterward;
then by the triangle 
inequality, $|x_1|$ and $|x_2-x_1|$ are at least $D-2\eta$.

Define $v=\frac{|x_1|}{|x_2-x_1|}(x_2-x_1)$; then
\[|x_1+v-x_2|=||x_1|-|x_2-x_1||\leqslant 2\eta\]
and
\[|x_1+v|\geqslant |x_2|-|x_1+v-x_2|\geqslant 2D-4\eta.\]

The vectors $x_1/|x_1|$ and $v/|x_1|$ have unit
norm and their average has norm at least $1-2\eta/D$; letting
$\varepsilon = ({2\eta}/cD)^{\frac1p}$ the convexity assumption therefore
yields $|x_1-v|\leqslant \varepsilon D$.
It follows that
\[|2x_1-x_2|\leqslant |x_1+v-x_2|+|x_1-v|\leqslant 2\eta+\varepsilon D.\]

Suppose that also $|x'_2|\geqslant 2(D-\eta)$; then the same reasoning
yields the same estimate on $|x'_2-2x_1|$ so that
\[|x_2-x'_2|\leqslant 4\eta+2 D \left(\frac{2\eta}{cD}\right)^{\frac1p}\]
Now we can choose $\eta=K/D^{p-1}$ with $K$ small enough to ensure that 
the above inequality reads $|x_2-x'_2|<2$. 
This contradicts the hypothesis that $\varphi$ is distance 
non-decreasing, therefore as desired $|x_2|$ or $|x'_2|$ must be smaller than $2(D-\eta)$.
\end{proof}

\subsection*{Acknowledgments}
I am grateful to Assaf Naor for interesting comments and pointers to the literature.

\bibliographystyle{hamsplain}
\bibliography{biblio}

\providecommand{\bysame}{\leavevmode\hbox to3em{\hrulefill}\thinspace}
\providecommand{\MR}{\relax\ifhmode\unskip\space\fi MR }
\providecommand{\MRhref}[2]{%
  \href{http://www.ams.org/mathscinet-getitem?mr=#1}{#2}
}
\providecommand{\href}[2]{#2}
\providecommand{\eprint}{\begingroup \urlstyle{tt}\Url}
\begin{thebibliography}{1}

\bibitem{Bourgain}
Jean Bourgain, \emph{The metrical interpretation of superreflexivity in
  {B}anach spaces}, Israel J. Math. \textbf{56} (1986), no.~2, 222--230.
  \MR{880292 (88e:46007)}

\bibitem{Johnson-Schechtman}
William~B. Johnson and Gideon Schechtman, \emph{Diamond graphs and
  super-reflexivity}, J. Topol. Anal. \textbf{1} (2009), no.~2, 177--189.
  \MR{2541760 (2010k:52031)}

\bibitem{Matousek}
Ji{\v{r}}{\'{\i}} Matou{\v{s}}ek, \emph{On embedding trees into uniformly
  convex {B}anach spaces}, Israel J. Math. \textbf{114} (1999), 221--237.
  \MR{1738681 (2001b:46028)}

\bibitem{Mendel-Naor}
Manor Mendel and Assaf Naor, \emph{Markov convexity and local rigidity of
  distorted metrics}, J. Eur. Math. Soc. (JEMS) \textbf{15} (2013), no.~1,
  287--337. \MR{2998836}

\end{thebibliography}

\end{document}